\documentclass[11pt]{article}
\usepackage{amsmath}
\usepackage{amssymb}
\usepackage{amsthm}
\usepackage[usenames]{color}
\usepackage{amscd}
\usepackage{dsfont}
\usepackage{indentfirst}

\usepackage[colorlinks=true,linkcolor=blue,filecolor=red,
citecolor=webgreen]{hyperref}
\definecolor{webgreen}{rgb}{0,.5,0}

\numberwithin{equation}{section}

\def\C{{\mathds{C}}}

\def\R{{\mathbb{R}}}
\def\N{{\mathds{N}}}
\def\Z{{\mathds{Z}}}

\def\1{{\bf 1}}

\def\pont{$\bullet$ }
\newcommand{\DOT}{\text{\rm\Huge{.}}}

\newtheorem{theorem}{Theorem}[section]

\newtheorem{lemma}[theorem]{Lemma}
\newtheorem{cor}[theorem]{Corollary}

\begin{document}

\title{{\bf Menon-type identities concerning additive characters}}
\author{L\'aszl\'o T\'oth \\ \\ Department of Mathematics, University of P\'ecs \\
Ifj\'us\'ag \'utja 6, 7624 P\'ecs, Hungary \\ E-mail: {\tt ltoth@gamma.ttk.pte.hu}}
\date{}
\maketitle

\centerline{Arabian Journal of Mathematics {\bf 9} (2020), 697--705}

\begin{abstract} By considering even functions (mod $n$) we generalize a recent Menon-type identity by Li and Kim, involving additive characters of
the group $\Z_n$. We use a different approach, based on certain convolutional identities. Some other applications, including related formulas
for Ramanujan sums, are discussed, as well.
\end{abstract}



\section{Introduction}

Menon's identity \cite{Men1965} states that for every $n\in \N$,
\begin{equation} \label{Menon_id}
\sum_{\substack{a=1\\ (a,n)=1}}^n (a-1,n) = \varphi(n) \tau(n),
\end{equation}
the notations used here and throughout the paper being fixed in Section \ref{Sect_Notations}. There are many generalizations and analogs of
identity \eqref{Menon_id} in the literature. See, e.g., the papers \cite{LiKimJNT,Tot2011,Tot2013,Tot2018,Tot,ZhaCao} and the references therein.

Let $\chi$ be a Dirichlet character (mod $n$) with conductor $d$, where $n,d\in \N$, $d\mid n$. Zhao and Cao \cite{ZhaCao}
derived the formula
\begin{equation} \label{Menon_id_char}
\sum_{a=1}^n (a-1,n) \chi(a)= \varphi(n) \tau(n/d),
\end{equation}
which recovers \eqref{Menon_id} if $\chi$ is the principal character (mod $n$), that is $d=1$.

The author \cite[Th.\ 2.4]{Tot2018} generalized identity \eqref{Menon_id_char} into
\begin{equation} \label{id_Toth}
\sum_{a=1}^n f_n(a-s) \chi(a) = \varphi(n) \chi^*(s)  \sum_{\substack{\delta \mid n/d\\ (\delta,s)=1}}
\frac{(\mu*f_n)(\delta d)}{\varphi(\delta d)},
\end{equation}
where $f_n$ is an even function (mod $n$), $s\in \Z$, and $\chi^*$ is the primitive character (mod $d$) that induces $\chi$.
We recall that a function $f_n:\Z \to \C$, $a\mapsto f_n(a)$ is said to be an even function (mod $n$) if
$f_n((a,n))=f_n(a)$ holds for every $a\in \Z$, where $n\in \N$ is fixed. Examples of even functions (mod $n$) are $f_n(a)=(a,n)$, more generally
$f_n(a)=F((a,n))$, where $F:\N \to \C$ is an arbitrary arithmetic function, and $f_n(a)=c_n(a)$, representing Ramanujan's sum.

In the case $f_n(a)=c_n(a)$, \eqref{id_Toth} gives (\cite[Cor.\ 2.6]{Tot2018})
\begin{equation} \label{c_n_chi}
\sum_{a=1}^n c_n(a-s) \chi(a) = d \varphi(n) \chi^*(s)  \sum_{\substack{\delta \mid n/d\\ (\delta,s)=1}}
\frac{\delta \mu(n/(\delta d))}{\varphi(\delta d)}.
\end{equation}

Notice that in \eqref{Menon_id_char}, \eqref{id_Toth} and \eqref{c_n_chi} the sums on the left hand sides are, in fact,
over $1\le a\le n$ with $(a,n)=1$, since
$\chi(a)=0$ for $(a,n)>1$. Here \eqref{c_n_chi} is a generalization of the first identity, due to Cohen \cite[Eq.\ (5.1)]{Coh1959}, of formulas
\begin{equation} \label{c_ident}
\sum_{\substack{a=1\\ (a,n)=1}}^n c_n(a-s) = \varphi(n) \sum_{\substack{\delta \mid n \\ (\delta,s)=1}}
\frac{\delta \mu(n/\delta)}{\varphi(\delta)} = \mu(n) c_n(s),
\end{equation}
the second formula of \eqref{c_ident} being the Brauer-Rademacher identity. See \cite[Cor.\ 34]{Coh1959} and \cite[Ch.\ 2]{McC1986}.

Recently, Li and Kim \cite{LiKimJNT} investigated the sums
\begin{equation*}
S(n,k):= \sum_{a=1}^n (a-1,n) e(ak/n)
\end{equation*}
and
\begin{equation} \label{def_sum_S_star}
S^*(n,k):= \sum_{\substack{a=1\\(a,n)=1}}^n (a-1,n) e(ak/n),
\end{equation}
by considering the additive characters $a\mapsto e(ak/n):= e^{2\pi iak/n}$ of the group $\Z_n$.
According to \cite[Th.\ 2.1]{LiKimJNT}, for every $n\in \N$ and
every $k\in \Z$ one has the identity
\begin{equation*}
S(n,k)= e(k/n) \sum_{\ell \mid (n,k)} \ell \varphi(n/\ell),
\end{equation*}
previously known in the literature, as mentioned by the authors, which is related to the discrete Fourier transform (DFT) of the gcd
function.

In order to investigate the sums $S^*(n,k)$ defined by \eqref{def_sum_S_star}, Li and Kim \cite{LiKimJNT} showed first that these sums
enjoy the modified multiplicativity property
\begin{equation*}
S^*(n_1n_2,k)= S^*(n_1,kn_2') S^*(n_2,kn_1')
\end{equation*}
for every $n_1,n_2\in \N$ with $(n_1,n_2)=1$, where $n_1n_1'\equiv 1$ (mod $n_2$), $n_2n_2'\equiv 1$ (mod $n_1$). See \cite[Prop.\ 3.1]{LiKimJNT}.
Then they computed the values $S^*(p^m,k)$ for prime powers $p^m$, and deduced that for every $n\in \N$, $k\in \Z$ such that
$\nu_p(n)-\nu_p(k)\ne 1$ for every prime $p\mid n$ one has (\cite[Th.\ 3.3]{LiKimJNT})
\begin{equation} \label{id_Li_Kim}
S^*(n,k)= e(k/n) \varphi(n) \tau((n,k)),
\end{equation}
which reduces to \eqref{Menon_id} in the case $n\mid k$. Furthermore, Li and Kim \cite[Th.\ 3.5]{LiKimJNT} established a formula for $S^*(n,k)$,
valid for every $n\in \N$ and $k\in \Z$.

In this paper we consider the following generalization of the sum $S^*(n,k)$:
\begin{equation} \label{def_T_f}
S_f(n,k,s):= \sum_{\substack{a=1\\ (a,n)=1}}^n f_n(a-s) e(ak/n),
\end{equation}
where $f=f_n$ is an even function (mod $n$), $n\in \N$, $s,k\in \Z$.  If $f_n$ is the constant $1$ function, then \eqref{def_T_f} reduces to the
Ramanujan sum $c_n(k)$. The sum \eqref{def_T_f} can be viewed as the DFT of the function
$h_n$ attached to $f_n$ and defined by $h_n(a)=f_n(a-s)$ if $(a,n)=1$ and  $h_n(a)=0$ if $(a,n)>1$. Several properties of the DFT of
even functions (mod $n$) were discussed in our paper \cite{TotHau2011}. However, if the function $f_n$ is even (mod $n$), then $h_n$ does
not have this property, in general.

Our results generalize those by Li and Kim \cite{LiKimJNT}. We use a different approach, similar to our paper \cite{Tot2018},
based on certain convolutional identities valid for every $n\in \N$.

We deduce in Theorem \ref{Th_gen} identities for the sums $S_f(n,k,s)$, valid
for every $n\in \N$ and $k\in \Z$,  which are simpler than the corresponding formula in \cite{LiKimJNT}. Then we consider the cases when
$\nu_p(n)\ge \nu_p(k)+2$ for every $p\mid n$, respectively $\nu_p(n)\le \nu_p(k)$ for every $p\mid n$
(Theorems \ref{Th_spec} and \ref{Th_spec2}). Finally, by taking sequences $(f_n)_{n\in \N}$ of even functions (mod $n$) such that
$n\mapsto f_n(a)$ is multiplicative for every fixed $a\in \Z$, we give in Theorem \ref{Th_n_1_n_2} and in Corollaries \ref{Cor_m} and \ref{Cor_m_1}
direct generalizations of formula \eqref{id_Li_Kim}. Some other applications, including identities for Ramanujan sums are derived, as well.
For example, Theorem \ref{Th_gen} gives that for every $n\in \N,s\in \Z$,
\begin{equation} \label{id_Raman_k_1}
\sum_{\substack{a=1\\ (a,n)=1}}^n c_n(a-s) e(a/n) = \sum_{\substack{d\delta=n \\(d, \delta s)=1}} d \mu^2(\delta)
e(\delta \delta'/n),
\end{equation}
where $\delta' \in \Z$ is such that $\delta \delta' \equiv s$ (mod $d$). If $n\in \N$ is squarefull (that is, $\nu_p(n)\ge 2$ for every prime
$p\mid n$) and $(s,n)=1$, then part iii) of Corollary \ref{Cor_m_1} shows that the sums in \eqref{id_Raman_k_1} equal
$n\, e(s/n)$. These identities may be compared to \eqref{c_ident}.

We point out that the key results, used in the proofs, are those concerning the sums $T_n(k,s,d)$, defined by \eqref{def_T}, which are themselves
generalizations of the Ramanujan sums.

\section{Notations} \label{Sect_Notations}

We use the following notations:

\pont $\N=\{1,2,\ldots\}$, $\Z$ is the set of integers,

\pont the prime power factorization of $n\in \Z$ is $n=\pm \prod_p p^{\nu_p(n)}$, the product being over the primes $p$, where all but
a finite number of the exponents $\nu_p(n)$ are zero, with the convention $\nu_p(0)=\infty$ for every prime $p$,

\pont $(a,b)$ denotes the greatest common divisor of $a,b\in \Z$,

\pont $(f*g)(n)=\sum_{d\mid n} f(d)g(n/d)$ is the Dirichlet convolution of the functions $f,g:\N \to \C$,

\pont $\mu$ is the M\"obius function,

\pont $\sigma_s(n)=\sum_{d\mid n} d^s$ ($s\in \R$),

\pont $\tau(n)=\sigma_0(n)$ is the number of divisors of $n$,

\pont $\sigma(n)=\sigma_1(n)$ is the sum of divisors of $n$,

\pont $J_s$ is the Jordan function of order $s$ given by
$J_s(n)=n^s\prod_{p\mid n} (1-1/p^s)$ ($s\in \R$),

\pont $\varphi=J_1$ is Euler's totient function,

\pont $e(x)=e^{2\pi i x}$ ($x\in \R$),

\pont $c_n(k)=\sum_{1\le a \le n, (a,n)=1} e(ak/n)$ is Ramanujan's sum ($n\in \N$, $k\in \Z$).

\section{Main results}

\begin{theorem} \label{Th_gen} Let $n\in \N$, $s,k\in \Z$ and let $f=f_n$ be an even function \textup{(mod $n$)}. Then
\begin{equation}
S_f(n,k,s) =  n \sum_{\substack{d\mid n\\ (d,s)=1}} \frac{(\mu*f_n)(d)}{d}  \sum_{\substack{\delta \mid n \\
(\delta,d)=1\\ \frac{n}{d\delta}\mid k}}
\frac{\mu(\delta)}{\delta} e(\delta \delta' k/n) \label{id_1}
\end{equation}
\begin{equation*}
=  \sum_{e \mid (n,k)} e \sum_{\substack{d\delta=n/e \\(d, \delta s)=1}} (\mu*f_n)(d) \mu(\delta)
e(\delta \delta' k/n),
\end{equation*}
where $\delta' \in \Z$ is such that $\delta \delta' \equiv s$ \textup{(mod $d$)}.
\end{theorem}

For $n\mid k$, formula \eqref{id_1} is a special case of \eqref{id_Toth}.

\begin{theorem} \label{Th_spec} Under the assumptions of Theorem \ref{Th_gen} and with $\nu_p(n)\ge \nu_p(k)+2$ for every prime $p\mid n$,
we have
\begin{equation} \label{id_spec}
S_f(n,k,s) = e(ks/n) \sum_{\substack{d\mid (n,k)\\ (n/d,s)=1}} d\, (\mu*f_n)(n/d).
\end{equation}
\end{theorem}

If $\nu_p(n)\le \nu_p(k)$ for every prime $p\mid n$ (that is, $n\mid k$), then $e(ak/n)=1$ for every integer $a$ and
\begin{equation*}
S_f(n,k,s) = \sum_{\substack{a=1\\ (a,n)=1}}^n f_n(a-s).
\end{equation*}

As a direct generalization of Menon's identity \eqref{Menon_id} we have the following result.

\begin{theorem} \label{Th_spec2} Under the assumptions of Theorem \ref{Th_gen} and with $\nu_p(n)\le \nu_p(k)$ for every prime $p\mid n$
one has
\begin{equation} \label{form_k_null}
S_f(n,k,s) = \varphi(n) \sum_{\substack{d\mid n\\ (d,s)=1}} \frac{(\mu*f_n)(d)}{\varphi(d)}.
\end{equation}
\end{theorem}

Formula \eqref{form_k_null} is, in fact, a special case of our result \cite[Th.\ 2.1]{Tot2018}, i.e., the case when $\chi$
is the principal character (mod $n$) in \eqref{id_Toth}.

\begin{cor} \label{Remark_spec_cases} Theorems \ref{Th_gen}, \ref{Th_spec} and \ref{Th_spec2}
apply to the following functions:

i) $f_n(a)=(a,n)^m$ \textup{($m\in \R$)} with $(\mu*f_n)(d)=J_m(d)$ \textup{($d\mid n$)},

ii) $f_n(a)=\sigma_m((a,n))$ \textup{($m\in \R$)} with $(\mu*f_n)(d)=d^m$ \textup{($d\mid n$)},

iii) $f_n(a)=c_n(a)$ with $(\mu*f_n)(d)=d \mu(n/d)$ \textup{($d\mid n$)}.
\end{cor}

Note that Theorem \ref{Th_spec2} applied with i) of Corollary \ref{Remark_spec_cases} and $m=1$ yields
\begin{equation*}
\sum_{\substack{a=1\\ (a,n)=1}}^n (a-s,n) = \varphi(n) \tau(n'),
\end{equation*}
where $n=n'n''$, $n'=\max \{d: d\mid n, (d,s)=1\}$. This generalizes Menon's identity \eqref{Menon_id} to any $n\in \N$ and $s\in \Z$.
Also see \cite[Eq.\ (35)]{Tot2011}.

See, e.g., \cite[Appl.\ 1]{TotHau2011} for the identity $\sum_{ab=d} \mu(a)c_n(b)= d\mu(n/d)$ ($d\mid n$), mentioned  in iii) of
Corollary \ref{Remark_spec_cases}. Several other special cases can be considered, as well.

\begin{theorem} \label{Th_seq_multipl} Let $(f_n)_{n\in \N}$ be a sequence of even functions \textup{(mod $n$)} such that
$n\mapsto f_n(a)$ is multiplicative for every fixed $a\in \Z$. Let $n_1,n_2\in \N$, $(n_1,n_2)=1$, $s,k\in \Z$.  Then
\begin{equation*}
S_f(n_1n_2,k,s)=S_f(n_1,kn_2',s) S_f(n_2,kn_1',s),
\end{equation*}
where $n_1',n_2'\in \Z$ are such that $n_1n_1'\equiv 1$ \textup{(mod $n_2$)} and $n_2n_2'\equiv 1$ \textup{(mod $n_1$)}.
\end{theorem}

Note that under assumptions of Theorem \ref{Th_seq_multipl}, $f_n(a)$ is multiplicative viewed as a function of two variables. See
\cite[Prop. 4]{TotHau2011}. Examples of sequences of functions satisfying the assumptions of Theorem \ref{Th_seq_multipl} are the
sequence of the Ramanujan sums $(c_n(\DOT))_{n\in \N}$ and $(f_n)_{n\in \N}$, where $f_n(a)=F((a,n))$ and $F$ is an arbitrary multiplicative
function.

\begin{theorem} \label{Th_n_1_n_2} Under the assumptions of Theorem \ref{Th_seq_multipl} and if $\nu_p(n)- \nu_p(k)\ne 1$ for every
prime $p\mid n$, we have
\begin{equation*}
S_f(n,k,s) = e(ks/n) \varphi(n_2) \sum_{\substack{d\mid (n_1,k)\\ (n_1/d,s)=1}} d\, (\mu*f_{n_1})(n_1/d)
\sum_{\substack{d\mid n_2\\ (d,s)=1}} \frac{(\mu*f_{n_2})(d)}{\varphi(d)},
\end{equation*}
where $n=n_1n_2$ such that $\nu_p(n)\ge \nu_p(k)+2$ for every prime $p\mid n_1$, and $\nu_p(n)\le \nu_p(k)$ for every prime
$p\mid n_2$.
\end{theorem}

\begin{cor} \label{Cor_m} Let $n\in \N$, $s,k\in \Z$ such that $\nu_p(n)- \nu_p(k)\ne 1$ for every prime $p\mid n$. Let $m\in \R$. Then
\begin{align} \label{11}
\sum_{\substack{a=1\\ (a,n)=1}}^n (a-s,n)^m e(ak/n) =  e(ks/n) J_m(n_1) \varphi(n_2) \sum_{\substack{d\mid (n_1,k)\\ (n_1/d,s)=1}}
d^{1-m}  \sum_{\substack{d\mid n_2 \\ (d,s)=1}} \frac{J_m(d)}{\varphi(d)},
\end{align}
\begin{align} \label{12}
\sum_{\substack{a=1\\ (a,n)=1}}^n \sigma_m((a-s,n)) e(ak/n) = e(ks/n) n_1^m  \varphi(n_2)
\sum_{\substack{d\mid (n_1,k)\\ (n_1/d,s)=1}}  d^{1-m} \sum_{\substack{d\mid n_2 \\ (d,s)=1}} \frac{d^m}{\varphi(d)},
\end{align}
\begin{align} \label{13}
\sum_{\substack{a=1\\ (a,n)=1}}^n c_n(a-s) e(ak/n) =  e(ks/n) n_1 \varphi(n_2) \sum_{\substack{d\mid (n_1,k)\\ (n_1/d,s)=1}}
\mu(d)  \sum_{\substack{d\mid n_2 \\ (d,s)=1}} \frac{d\mu(n_2/d)}{\varphi(d)},
\end{align}
where $n_1$ and $n_2$ are defined as in Theorem \ref{Th_n_1_n_2}.
\end{cor}

Next we consider the case when $(s,n)=1$.

\begin{cor} \label{Cor_m_1} Let $n\in \N$, $s,k\in \Z$ such that $(s,n)=1$ and $\nu_p(n)- \nu_p(k)\ne 1$ for every prime $p\mid n$.
Let $n_1$ and $n_2$ be defined as in Theorem \ref{Th_n_1_n_2}.

i) If $m\in \R$, $m\ne 1$, then
\begin{align*}
\sum_{\substack{a=1\\ (a,n)=1}}^n (a-s,n)^m e(ak/n) =  e(ks/n) J_m(n_1)  \sigma_{1-m}((n_1,k)) \varphi(n_2) F_m(n_2),
\end{align*}
where
\begin{align*}
F_m(n_2):=\sum_{d\mid n_2} \frac{J_m(d)}{\varphi(d)}= \prod_{p^{\nu_p(n_2)}\mid \mid n_2}
\left(1+\frac{p^m-1}{p-1} \cdot\frac{p^{(m-1)\nu_p(n_2)}-1}{p^{m-1}-1}\right).
\end{align*}

For $m=1$,
\begin{align} \label{1111}
\sum_{\substack{a=1\\ (a,n)=1}}^n (a-s,n) e(ak/n) = e(ks/n) \varphi(n)\tau((n,k)).
\end{align}

ii) If $m\in \R$, $m\ne 1$, then
\begin{align*}
\sum_{\substack{a=1\\ (a,n)=1}}^n \sigma_m((a-s,n)) e(ak/n) =  e(ks/n) n_1^m
\sigma_{1-m}((n_1,k))\varphi(n_2) G_m(n_2),
\end{align*}
where
\begin{align*}
G_m(n_2): =\sum_{d\mid n_2} \frac{d^m}{\varphi(d)}= \prod_{p^{\nu_p(n_2)}\mid \mid n_2}
\left(1+\frac{p^m}{p-1} \cdot \frac{p^{(m-1)\nu_p(n_2)}-1}{p^{m-1}-1}\right).
\end{align*}

For $m=1$,
\begin{align*}
\sum_{\substack{a=1\\ (a,n)=1}}^n \sigma((a-s,n)) e(ak/n) = e(ks/n) n_1
\tau((n_1,k))\varphi(n_2) G_1(n_2),
\end{align*}
where
\begin{align*}
G_1(n_2): =\sum_{d\mid n_2} \frac{d}{\varphi(d)}= \prod_{p^{\nu_p(n_2)}\mid \mid n_2}
\left(1+ \frac{p\nu_p(n_2)}{p-1}\right).
\end{align*}

iii) We have
\begin{equation*}
\sum_{\substack{a=1\\ (a,n)=1}}^n c_n(a-s) e(ak/n) =  \begin{cases} e(ks/n) n_1,  & \text{ if $(n_1,k)=1$ and $n_2$ is squarefree}, \\
0, & \text{ otherwise}. \end{cases}
\end{equation*}
\end{cor}

In the case $s=1$, \eqref{1111} reduces to the identity \eqref{id_Li_Kim} by Li and Kim.

\section{Proofs}

We need the following lemmas. For $n,d\in \N$ and $k,s\in \Z$ consider the sum
\begin{equation} \label{def_T}
T_n(k,s,d):=\sum_{\substack{a=1\\ (a,n)=1\\ a\equiv s \text{\rm (mod $d$)} }}^n e(ak/n),
\end{equation}
which reduces to Ramanujan's sum $c_n(k)$ if $d=1$.

\begin{lemma} \label{Lemma_Ramanujan_gen} Let $n,d\in \N$, $k,s\in \Z$ such that $d\mid n$. Then
\begin{equation*}
T_n(k,s,d)=
\begin{cases} \displaystyle \frac{n}{d} \sum_{\substack{\delta \mid n \\ (\delta,d)=1\\ \frac{n}{d\delta}\mid k}} \frac{\mu(\delta)}{\delta}
e(\delta \delta' k/n),
& \text{ if $(s,d)=1$}, \\ 0, & \text{ otherwise},
\end{cases}
\end{equation*}
where $\delta' \in \Z$ is such that $\delta \delta' \equiv s$ \textup{(mod $d$)}.
\end{lemma}

Note that in the case $d=1$, this recovers the familiar formula
\begin{equation*}
c_n(k)= \sum_{\delta \mid (n,k)} \delta \mu(n/\delta),
\end{equation*}
concerning Ramanujan's sum.

\begin{proof}[Proof of Lemma {\rm \ref{Lemma_Ramanujan_gen}}]
For each term of the sum, $(a,n)=1$. Hence if $a\equiv s$ (mod $d$), then $(s,d)=(a,d)=1$. We assume that $(s,d)=1$ is satisfied
(otherwise the sum is empty and equals zero).

Using the property of the M\"{o}bius $\mu$ function, the given sum can be written as
\begin{equation*}
T_n(k,s,d) = \sum_{\substack{a=1\\ a\equiv s \text{\rm (mod $d$)} }}^n e(ak/n) \sum_{\delta \mid (a,n)} \mu(\delta)
\end{equation*}
\begin{equation*}
= \sum_{\delta \mid n} \mu(\delta) \sum_{\substack{j=1\\ \delta j\equiv s \text{\rm (mod $d$)} }}^{n/\delta} e(\delta jk/n).
\end{equation*}

Let $\delta\mid n$ be fixed. The linear congruence $\delta j\equiv s$ (mod $d$) has solutions in $j$ if and only if $(\delta,d)\mid s$, equivalent to
$(\delta,d)=1$, since $(s,d)=1$. There are $n/(d\delta)$ solutions $j$ (mod $n/\delta$) and these are $j=\delta'+\ell d$, where
$1\le \ell \le n/(d\delta)$ and $\delta'$ is any integer such that $\delta \delta' \equiv s$ \textup{(mod $d$)}.

We deduce that
\begin{equation*}
T_n(k,s,d) = \sum_{\substack{\delta \mid n\\ (\delta,d)=1}} \mu(\delta) \sum_{\ell=1}^{n/(d\delta)} e(\delta(\delta'+\ell d) k/n)
\end{equation*}
\begin{equation*}
= \sum_{\substack{\delta \mid n\\ (\delta,d)=1}} \mu(\delta) e(\delta \delta' k/n) \sum_{\ell=1}^{n/(d\delta)} e(\ell k/(n/d\delta)),
\end{equation*}
where the inner sum is $n/(d\delta)$ if $n/(d\delta)\mid k$ and is $0$ otherwise. This gives the formula.
\end{proof}

\begin{lemma} \label{Lemma_Ramanujan_spec}
Let $n,d\in \N$, $k,s\in \Z$ such that $d\mid n$ and $\nu_p(n)\ge \nu_p(k)+2$ for all primes $p\mid n$. Then
\begin{equation*}
T_n(k,s,d) = \begin{cases} \displaystyle \frac{n}{d} e(ks/n),
& \text{ if $\frac{n}{d} \mid k$ and $(s,d)=1$}, \\ 0, & \text{ otherwise}.
\end{cases}
\end{equation*}
\end{lemma}

\begin{proof}[Proof of Lemma {\rm \ref{Lemma_Ramanujan_spec}}]
We use Lemma \ref{Lemma_Ramanujan_gen}. Assume that $(s,d)=1$. In the sum it is enough to consider the squarefree values of $\delta$,
that is $\nu_p(\delta)\le 1$
for all primes $p\mid n$, since otherwise $\mu(\delta)=0$. We claim that the only possible value is $\delta=1$, that is $\nu_p(\delta)=0$
for all primes $p\mid n$. Indeed, if there is a prime $p$ such that $\nu_p(\delta)=1$, then $\nu_p(d)=0$, since $(d,\delta)=1$, and by the
assumption $\nu_p(n)\ge \nu_p(k)+2$ we get $\nu_p(\frac{n}{d\delta})=\nu_p(n)-1\ge \nu_p(k)+1>\nu_p(k)$, which contradicts the condition
$\frac{n}{d\delta}\mid k$.

Now for $\delta=1$ one can select $\delta'=s$. This completes the proof.
\end{proof}

\begin{lemma} \label{Lemma_Ramanujan_spec_2}
Let $n,d\in \N$, $k,s\in \Z$ such that $d\mid n$ and $\nu_p(n)\le \nu_p(k)$ for all primes $p\mid n$
\textup{(that is, $n\mid k$, in particular $k=0$)}. Then
\begin{equation*}
T_n(k,s,d)= \sum_{\substack{a=1\\ (a,n)=1\\ a\equiv s \text{\rm (mod $d$)} }}^n 1  =
\begin{cases} \displaystyle \frac{\varphi(n)}{\varphi(d)}, & \text{ if $(s,d)=1$}, \\ 0, & \text{ otherwise},
\end{cases}
\end{equation*}
\end{lemma}

\begin{proof}[Proof of Lemma {\rm \ref{Lemma_Ramanujan_spec_2}}]
Use Lemma \ref{Lemma_Ramanujan_gen}. Assume that $(s,d)=1$. Since $n\mid k$, in the sum $e(\delta \delta'k/n)=1$ for every $\delta$,
and we have
\begin{equation*}
T_n(k,s,d)= \frac{n}{d} \sum_{\substack{\delta \mid n\\ (\delta,d)=1}} \frac{\mu(\delta)}{\delta}= \frac{n}{d} \prod_{\substack{p\mid n \\ p\nmid d}}
\left(1-\frac1{p}\right)=\frac{n}{d}\cdot \frac{\varphi(n)/n}{\varphi(d)/d}=\frac{\varphi(n)}{\varphi(d)}.
\end{equation*}
\end{proof}

Lemma \ref{Lemma_Ramanujan_spec_2} is known in the literature, and is usually proved by the inclusion-exclusion principle.
See, e.g., \cite[Th.\ 5.32]{Apo1976}.

\begin{proof}[Proof of Theorem {\rm \ref{Th_gen}}]
Since the function $f_n$ is even (mod $n$), we have for every $a\in \N$,
\begin{equation*}
f_n(a)=f_n((a,n))= \sum_{d\mid (a,n)} (\mu*f_n)(d).
\end{equation*}

Hence,
\begin{equation*}
S_f(n,k,s)= \sum_{\substack{a=1\\ (a,n)=1}}^n e(ak/n) \sum_{\substack{d\mid (a-s,n)}} (\mu*f_n)(d)
\end{equation*}
\begin{equation*}
= \sum_{d\mid n} (\mu*f_n)(d) \sum_{\substack{a=1\\ (a,n)=1\\ a\equiv s \text{\rm (mod $d$)}}}^n  e(ak/n),
\end{equation*}
that is,
\begin{equation} \label{4}
S_f(n,k,s) = \sum_{d\mid n} (\mu*f_n)(d) T_n(k,s,d).
\end{equation}

According to Lemma \ref{Lemma_Ramanujan_gen} we deduce
\begin{equation*}
S_f(n,k,s)= n \sum_{\substack{d\mid n\\ (d,s)=1}} \frac{(\mu*f_n)(d)}{d}  \sum_{\substack{\delta \mid n \\ (\delta,d)=1\\ \frac{n}{d\delta}\mid k}}
\frac{\mu(\delta)}{\delta} e(\delta \delta' k/n),
\end{equation*}
\begin{equation*}
= n \sum_{\substack{d\delta e=n\\ (d,\delta s)=1\\ e\mid k}} \frac{(\mu*f_n)(d)}{d} \cdot \frac{\mu(\delta)}{\delta} e(\delta \delta' k/n),
\end{equation*}
which gives the result.
\end{proof}

\begin{proof}[Proof of Theorem {\rm \ref{Th_spec}}]
By using identity \eqref{4} and Lemma \ref{Lemma_Ramanujan_spec} we have
\begin{equation*}
S_f(n,k,s)= n e(ks/n) \sum_{\substack{d\mid n \\ \frac{n}{d} \mid k\\ (d,s)=1}}
\frac{(\mu*f_n)(d)}{d},
\end{equation*}
and interchange $d$ and $n/d$. This gives identity \eqref{id_spec}.
\end{proof}

\begin{proof}[Proof of Theorem {\rm \ref{Th_spec2}}]
Follows immediately by identity \eqref{4} and Lemma \ref{Lemma_Ramanujan_spec_2}.
\end{proof}

\begin{proof}[Proof of Theorem {\rm \ref{Th_seq_multipl}}]
Since $(n_1,n_2)=1$, if $a_1$ runs over a reduced residue system (mod $n_1$) and $a_2$ runs through a reduced residue system (mod $n_2$), then
$a=a_1n_2+a_2n_1$  runs through a reduced residue system (mod $n_1n_2$). Hence
\begin{equation*}
S_f(n_1n_2,k,s)= \sum_{\substack{a_1=1\\ (a_1,n_1)=1}}^{n_1} \sum_{\substack{a_2=1\\ (a_2,n_2)=1}}^{n_2} f_{n_1n_2}(a_1n_2+a_2n_1-s)
e((a_1n_2+a_2n_1)k/(n_1n_2)).
\end{equation*}

Using that $f_n$ are even functions (mod $n$) and that $n\mapsto f_n(a)$ is multiplicative for every $a\in \Z$ we have
\begin{equation*}
f_{n_1n_2}(a_1n_2+a_2n_1-s) =  f_{n_1}(a_1n_2+a_2n_1-s) f_{n_2}(a_1n_2+a_2n_1-s)
\end{equation*}
\begin{equation*}
= f_{n_1}((a_1n_2+a_2n_1-s,n_1)) f_{n_2}((a_1n_2+a_2n_1-s,n_2))
\end{equation*}
\begin{equation*}
=f_{n_1}((a_1n_2-s,n_1)) f_{n_2}((a_2n_1-s,n_2)) = f_{n_1}(a_1n_2-s) f_{n_2}(a_2n_1-s)
\end{equation*}

Therefore,
\begin{equation*}
S_f(n_1n_2,k,s)= \sum_{\substack{a_1=1\\ (a_1,n_1)=1}}^{n_1} f_{n_1}(a_1n_2-s) e(a_1n_2n_2'k/n_1)
\end{equation*}
\begin{equation*}
\times \sum_{\substack{a_2=1\\ (a_2,n_2)=1}}^{n_2} f_{n_2}(a_2n_1-s) e(a_2n_1n_1'k/n_2)
= S_f(n_1,kn_2',s) S_f(n_2,kn_1',s),
\end{equation*}
where $a_1n_2$ runs through a reduced residue system (mod $n_1$) and $a_2n_1$ runs through a reduced residue
system (mod $n_2$).
\end{proof}

\begin{proof}[Proof of Theorem {\rm \ref{Th_n_1_n_2}}]
Using Theorems \ref{Th_seq_multipl}, \ref{Th_spec} and \ref{Th_spec2} we have
\begin{equation*}
S_f(n,k,s)= S_f(n_1n_2,k,s) = S_f(n_1,kn_2',s) S_f(n_2,kn_1',s)
\end{equation*}
\begin{equation*}
= e(kn_2's/n_1) \sum_{\substack{d\mid (n_1,kn_2')\\ (n_1/d,s)=1}} d\, (\mu*f_{n_1})(n_1/d) \varphi(n_2)
\sum_{\substack{d\mid n_2\\ (d,s)=1}} \frac{(\mu*f_{n_2})(d)}{\varphi(d)},
\end{equation*}
with
\begin{equation*}
e(kn_2's/n_1)= e(ks n_2 n_2'/n_1n_2)= e(ks(1+tn_1)/n_1n_2)
\end{equation*}
\begin{equation*}
=  e(ks/n) e(kst/n_2)= e(ks/n),
\end{equation*}
where $k/n_2\in \Z$ by the definition of $n_2$. Also, since $(n_2',n_1)=1$, one has $(n_1,kn_2')=(n_1,k)$.
\end{proof}

\begin{proof}[Proof of Corollary {\rm \ref{Cor_m}}]
Apply Theorem \ref{Th_n_1_n_2} for $f_n(a)=(a,n)^m$, $f_n(a)=\sigma_m((a,n))$ and $f_n(a)=c_n(a)$, respectively.

If $f_n(a)=(a,n)^m$, then $(\mu*f_n)(d)=J_m(d)$ for every $d\mid n$, as noted in Corollary \ref{Remark_spec_cases}. Also, for $d\mid (n_1,k)$ we have
$\nu_p(n_1)\ge \nu_p(k)+2\ge \nu_p(d)+2$ for every prime $p\mid n_1$ which implies that $J_m(n_1/d)= J_m(n_1)/d^m$, giving \eqref{11}.

If $f_n(a)=\sigma_m((a,n))$, then $(\mu*f_n)(d)= d^m$ for every $d\mid n$, leading to \eqref{12}

If $f_n(a)=c_n(a)$, then $(\mu*f_n)(d)=d \mu(n/d)$ for every $d\mid n$. See Corollary \ref{Remark_spec_cases}.  This implies \eqref{13}.
\end{proof}

\begin{proof}[Proof of Corollary {\rm \ref{Cor_m_1}}]
Follows immediately by Corollary \ref{Cor_m}.
\end{proof}

\section{Acknowledgement} This work was supported by the European Union, co-financed by the European
Social Fund EFOP-3.6.1.-16-2016-00004.


\begin{thebibliography}{99}

\bibitem{Apo1976} T.~M.~Apostol, {\it Introduction to Analytic Number Theory}, Springer, 1976.

\bibitem{Coh1959} E.~Cohen, Representations of even functions (mod $r$), II. Cauchy products, {\it Duke Math. J.}
{\bf 26} (1959), 165--182.

\bibitem{LiKimJNT} Y.~Li and D.~Kim, Menon-type identities with additive characters,
{\it J. Number Theory} {\bf 192} (2018), 373--385.

\bibitem{McC1986} P.~J.~McCarthy, {\it Introduction to Arithmetical Functions}, Springer, 1986.

\bibitem{Men1965} P.~K.~Menon, On the sum $\sum \,(a-1,\,n) [(a,\, n) = 1]$, {\it J. Indian
Math. Soc. (N.S.)} {\bf 29} (1965), 155--163.

\bibitem{Tot2011} L.~T\'oth, Menon's identity and arithmetical sums representing functions of several variables,
{\it Rend. Sem. Mat. Univ. Politec. Torino} {\bf 69} (2011), 97--110.

\bibitem{Tot2013} L.~T\'oth, Another generalization of the gcd-sum function, {\it Arab J. Math.} {\bf 2} (2013), 313--320.

\bibitem{Tot2018} L.~T\'oth, Menon-type identities concerning Dirichlet characters, {\it Int. J. Number Theory} {\bf 14} (2018),
1047--1054.

\bibitem{Tot} L.~T\'oth, Short proof and generalization of a Menon-type identity by Li, Hu and Kim, {\it Taiwanese J. Math.} {\bf 23} 
(2019), 557--561.

\bibitem{TotHau2011} L.~T\'oth  and P.~Haukkanen, The discrete Fourier transform of $r$-even functions, {\it Acta Univ. Sapientiae, Math.}
{\bf 3} (2011), 5--25.

\bibitem{ZhaCao} X.-P.~Zhao and Z.-F.~Cao, Another generalization of Menon's identity, {\it Int. J. Number Theory} {\bf 13} (2017),
2373--2379.

\end{thebibliography}
\end{document}